\documentclass[11pt]{amsart}

\usepackage{latexsym,amssymb,amsthm,amsmath,amscd,multirow,graphics,multicol}

\theoremstyle{plain}

\newtheorem{theorem}{Theorem}[section]
\newtheorem{lemma}{Lemma}[section]

\theoremstyle{remark}
\newtheorem{remark}{Remark}
\theoremstyle{remark}

\theoremstyle{definition}

\numberwithin{equation}{section}

\def\e{\eqref}
\def\p{\partial }
\def\<{\left < }
\def\>{\right >}
\def\({\left ( }
\def\){\right )}

\def\sech{\,{\rm sech\,}}
\def\e{\eqref}

\begin{document}

\title[ Gauss-Codazzi  and the Ricci equations]
{Dependence of the Gauss-Codazzi equations and the Ricci equation of Lorentz surfaces}

\author{BANG-YEN CHEN (East Lansing)}

\address{Department of Mathematics\\
    Michigan State University \\East Lansing, Michigan 48824--1027\\ U.S.A.}
\email{bychen@math.msu.edu}

 \subjclass[2000]{Primary: 53C40; Secondary  53C50}

\keywords{Lorentz surfaces;   equation of Ricci; equations of Gauss-Codazzi;  Lorentzian Kaehler surface.}

\begin{abstract}  The fundamental equations of Gauss, Codazzi and Ricci  provide the conditions for local isometric embeddability.  In general, the three fundamental equations  are independent for  surfaces in  Riemannian $4$-manifolds.  In contrast, we prove in this article  that for  arbitrary Lorentz surfaces  in  Lorentzian Kaehler surfaces the   equation of Ricci is a consequence of the equations of Gauss and Codazzi.   

\end{abstract}

\maketitle

\section{Introduction.}
Let $\tilde M^n$ be a complex $n$-dimensional indefinite Kaehler manifold, that means $\tilde M^n$ is
endowed with an almost complex structure $J$ and with an indefinite Riemannian
metric $\tilde g$, which is $J$-Hermitian, i.e., for all $p\in \tilde M^n$, we have
\begin{align}\label{1.1} &\tilde g(JX, JY)=\tilde g(X,Y), \;\; \forall  X,Y\in T_pM^n,  \\& \label{1.2}\tilde \nabla J=0, \end{align}
where $\tilde \nabla$ is the Levi-Civita connection of $\tilde g$. It follows  that $J$  is
integrable.

The complex index of $\tilde M^n$ is defined as the complex dimension of the largest complex negative
definite subspace of the tangent space. When the complex index is one,  we denote the  indefinite Kaehler manifold  by $\tilde M^n_1$, which is called a {\it Lorentzian Kaehler manifold} (cf.  \cite{br}).

The curvature tensor $\tilde  R$ of an  indefinite Kaehler manifold $\tilde M^n$ satisfies
\begin{align} & \label{1.3} \tilde R(X,Y;Z,W)=-\tilde R(Y,X;Z,W),
\\& \label{1.4}\tilde R(X,Y;Z,W) =\tilde R(Z,W;X,Y),
\\& \label{1.5} \tilde R(X,Y;JZ,W) =-\tilde R(X,Y;Z,JW),\end{align} where $\tilde R(X,Y;Z,W)=\tilde g(\tilde R(X,Y)Z,W)$.

It is well-known that the three fundamental equations of Gauss, Codazzi and Ricci play fundamental roles in the theory of submanifolds. For  surfaces in  Riemannian 4-manifolds, 
the three  equations of Gauss, Codazzi and Ricci are independent in general. 

On the other hand, we prove in this article a fundamental result for Lorentz surfaces; namely, {\it for any Lorentz surface  in any Lorentzian Kaehler surface the   equation of Ricci is a consequence of the equations of Gauss and Codazzi. }

\section{Basic formulas and fundamental equations}

Let $M^2_1$ be a Lorentz surface in a  Lorentzian Kaehler surface $\tilde M^2_1$ with an almost complex structure $J$ and  Lorentzian Kaehler metric $\tilde g$. 
Let  $g$ denote the induced metric on $M^2_1$. Denote by $\nabla$ and $\tilde\nabla$ the Levi-Civita connection on $g$ and $\tilde g$, respectively; and  by $R$  the curvature tensor of $M$.

 The formulas of Gauss and Weingarten are given respectively by (cf. \cite{c1,O})
\begin{align} &\label{2.1}\tilde \nabla_XY=\nabla_XY+h(X,Y),\\& \label{2.2}\tilde\nabla_X\xi=-A_\xi X+D_X \xi\end{align}
for vector fields $X,Y$ tangent to $M^2_1$ and $\xi$ normal to $M$,
where $h,A$ and $D$ are the second fundamental form, the shape operator and the normal connection. 

For a normal vector $\xi$ of $M^2_1$ at $x\in M^2_1$, the shape operator $A_{\xi}$ is a symmetric endomorphism of the tangent space $T_xM^2_1$.  The shape operator and the second fundamental form are related by
\begin{align}\label{2.3} \tilde g(h(X,Y),\xi)=g(A_{\xi}X,Y)\end{align}
for $X,Y$ tangent to $M^2_1$.

The three fundamental equations of Gauss, Codazzi and Ricci are given  by
\begin{align}\label{2.4} &R(X,Y;Z,W) =\tilde R(X,Y;Z,W)+  \<h(X,W),h(Y,Z)\>\\&\notag \hskip1.2in  - \<h(X,Z),h(Y,W)\>, \\ &  \label{2.5} (\tilde R(X,Y)Z)^\perp=(\bar\nabla_X h)(Y,Z) - (\bar\nabla_Y h)(X,Z),
\\&\label{2.6} \tilde g(R^D(X,Y)\xi,\eta)=\tilde R(X,Y;\xi,\eta)+ g([A_\xi,A_\eta]X,Y),
\end{align}
where $X,Y,Z,W$ are vector tangent to $M^2_1$,
and $\bar\nabla h$ is defined by
\begin{equation}\begin{aligned}\label{2.7}(\bar \nabla_X h)(Y,Z) = D_X h(Y,Z) - h(\nabla_X Y,Z) - h(Y,\nabla_XZ).\end{aligned}\end{equation}

The following lemma is an easy consequence of a result of \cite{L}.

\begin{lemma} \label{L:1} Locally there exists a coordinate system $\{x,y\}$ on  a Lorenz surface $M^2_1$ such that the metric tensor  is given by
\begin{align}\label{2.8}g=-m^2(x,y)^2(dx\otimes dy+dy\otimes dx)\end{align}
for some positive function $m(x,y)$.
\end{lemma}
\begin{proof} It is known  that locally there exist  isothermal coordinates $(u,v)$ on a  Lorentz surface $M^2_1$ such that the metric tensor  takes the form:
\begin{align}\label{2.9} g=E(u,v)^2(-du\otimes du +dv\otimes dv)\end{align} for some positive function $E$ (see \cite{L} (see, also \cite{CV}). Thus, after putting $$x=u+v,\quad y=u-v,$$  we obtain \e{2.8} from \e{2.9} with $m(x,y)=E(x,y)/\sqrt{2}$.
\end{proof}

\section{Main theorem.}

The main purpose of this article is prove the following fundamental result for Lorentz surfaces.

\begin{theorem}\label{T:1} The equation of Ricci is a consequence of  the equations of Gauss and Codazzi for any  Lorentz surface in any Lorentzian Kaehler surface.
\end{theorem}
\begin{proof} Assume that $\phi: M^2_1\to \tilde M_1^2$ is an isometric immersion of a  Lorentz  surface $M^2_1$ into  a Lorentzian Kaehler surface $\tilde M_1^2$. According to Lemma \ref{L:1}, we may assume that  locally  $M^2_1$ is equipped with the following Lorentzian metric:
\begin{align}\label{3.1} g=-m^2(x,y)(dx\otimes dy+dy\otimes dx)\end{align}
for some positive function $m$.
The Levi-Civita connection of $g$ satisfies
\begin{align}\label{3.2}\nabla_{\frac{\p}{\p x}}\frac{\p}{\p x}=\frac{2 m_x}{m}  \frac{\p}{\p x}, \; \nabla_{\frac{\p}{\p x}}\frac{\p}{\p y}=0, \;  \nabla_{\frac{\p}{\p y}}\frac{\p}{\p y}=\frac{2 m_y} {m} \frac{\p}{\p y}\end{align}
and the Gaussian curvature $K$  is given by
\begin{align}\label{3.3} K=\frac{2m m_{xy}-2m_x m_y}{m^4}.\end{align}

If we put \begin{align}\label{3.4} e_1=\frac{1}{m}\frac{\partial}{\partial x} ,\;\; e_2=\frac{1}{m}\frac{\partial}{\partial y},\end{align} then $\{e_1,e_2\}$ is a pseudo-orthonormal frame satisfying 
\begin{align} \label{3.5} & \<e_1, e_1\>=\< e_2,e_2\>=0,\;\<e_1, e_2\>=-1.\end{align} 

 From \e{3.2} and \e{3.4} we find
\begin{equation}\begin{aligned}\label{3.6}&\nabla_{e_1}e_1=\frac{m_x}{m^2}e_1, \; \nabla_{e_2}e_1=-\frac{m_y}{m^2}e_1, \; \\& \nabla_{e_1}e_2=-\frac{m_x}{m^2}e_2, \;  \nabla_{e_2}e_2=\frac{m_y}{m^2}e_2 .\end{aligned}\end{equation}

 For each tangent vector $X$ of $M^2_1$, we put
\begin{align} \label{3.7} &JX=PX+FX ,\end{align}
where $PX$ and $FX$ are the tangential and the normal components of $JX$. 
For the pseudo-orthonormal frame $\{e_1,e_2\}$ defined by \e{3.4}, it follows from \e{1.1},  \e{3.5}, and \e{3.7}  that 
\begin{align} \label{3.8} &Pe_1=(\sinh\alpha) e_1,\;\; Pe_2=-(\sinh \alpha) e_2\end{align} for some function $\alpha$. We call this function $\alpha$ the {\it Wirtinger angle}. 

If we put 
\begin{align} \label{3.9} &e_3=(\sech \alpha)Fe_1,\;\; e_4=(\sech \alpha)Fe_2,\end{align}
then we may derive from \e{3.7}-\e{3.9}  that
 \begin{align}& \label{3.10}J e_1=\sinh \alpha e_1+\cosh\alpha e_3,\hskip.2in Je_2=-\sinh\alpha e_2+\cosh\alpha e_4,\\& \label{3.11}Je_3=-\cosh \alpha e_1-\sinh \alpha e_3,\;\;  Je_4=-\cosh \alpha e_2+\sinh \alpha e_4,\\\label{3.12} &\<e_3,e_3\>=\<e_4,e_4\>=0,\;\; \<e_3,e_4\>=-1.\end{align}
   We call such a frame $\{e_1,e_2,e_3,e_4\}$ an {\it adapted pseudo-orthonormal frame} for  $M^2_1$. 
  
Let us put
$\nabla_X e_j=\sum_{k=1}^2 \omega_j^k(X)e_k;j,k=1,2
$. Then we deduce from \eqref{3.5} that 
\begin{align} &\label{3.13} \nabla_X e_1=\omega(X) e_1,\;\;  \nabla_X e_2=-\omega(X) e_2,\;\; \omega=\omega_1^1. \end{align}

Similarly, if we put
$D_Xe_r=\omega_r^s(X)e_s; r,s=3,4$, then \eqref{3.12} yields \begin{align} & \label{3.14} D_X e_3=\Phi(X) e_3,\; \; D_X e_4=-\Phi(X) e_4,\;\; \Phi=\omega_3^3. \end{align}

For the second fundamental form $h$,  we put $h(e_i,e_j)=h^3_{ij}e_3+h^4_{ij}e_4.$ Then,
by applying  $\tilde\nabla_X(JY)=J\tilde\nabla_XY$,  \e{3.10}-\e{3.14},  we may obtain the following:
\begin{align} & \label{3.15}  A_{e_3}e_j=h^4_{j2}e_1+h^4_{1j}e_2,\;\;  A_{e_4}e_j=h^3_{j2}e_1+h^3_{1j}e_2,\\&\label{3.16}e_j\alpha=(\omega_j- \Phi_j)\coth \alpha -2h^3_{1j},
\\ \label{3.17} &e_1\alpha=h^4_{12}-h^3_{11},\;\; e_2\alpha=h^4_{22}-h^3_{12},
\\&\label{3.18} \omega_j-\Phi_j=(h^3_{1j}+h^4_{j2})\tanh\alpha,
\end{align}  where $\omega_j=\omega(e_j)$ and $\Phi_j=\Phi(e_j)$ for $j=1,2$.

For simplicity, let us put
 \begin{align} \label{3.19} & h(e_1,e_1)=\beta e_3+\gamma e_4,\; h(e_1,e_2)=\delta e_3+\varphi e_4,\;\; h(e_2,e_2)=\lambda e_3+\mu e_4. \end{align} 
 In view of  \e{3.12}, and \e{3.19}, equation \e{2.4} of Gauss can be expressed as \begin{align}\label{3.20}\gamma\lambda+\beta \mu-2\delta\varphi=\frac{2(m m_{xy}-m_x m_y)}{m^4}-\tilde K,\end{align}
 where  $\tilde K=  -\tilde R(e_1,e_2;e_2,e_1) $  is the sectional curvature of the ambient space $\tilde M^2_1$ with respect to the 2-plane spanned by $e_1,e_2$.
 
By using \e{3.6}, \e{3.14}, and  \e{3.18}  we find
\begin{equation}\begin{aligned}\label{3.21}& D_{e_1}e_3=\(\frac{m_x}{m^2} -(\beta +\varphi)\tanh \alpha\)e_3, \;\\& D_{e_2}e_3=-\(\frac{m_y}{m^2} +(\delta+\mu) \tanh \alpha \)e_3, \\& D_{e_1}e_4=\((\beta+\varphi) \tanh \alpha-\frac{m_x}{m^{2}} \)e_4, \; \;\; \\&D_{e_2}e_4=\(\frac{m_y}{m^2}+(\delta+\mu) \tanh \alpha\)e_4.
\end{aligned}\end{equation}
So, it follows from \e{3.6}, \e{3.19} and \e{3.21} that
\begin{equation}\begin{aligned}\label{3.22}
 (\bar\nabla_{e_1}h)(e_1,e_1)=&\(\frac{\beta_x}{m}-\frac{\beta m_x}{m^2}-\beta(\beta+\varphi)\tanh\alpha  \)e_3 \\ \hskip.1in &+\(\frac{\gamma_x}{m}- \frac{3\gamma m_x}{m^2}+\gamma(\beta+\varphi)\tanh\alpha \)e_4,\\ 
(\bar\nabla_{e_1}h)(e_1,e_2)=&\(\frac{\delta_x}{m}+\frac{\delta m_x}{m^2}-\delta(\beta+\varphi)\tanh\alpha  \)e_3 \\& \hskip.1in +\(\frac{\varphi_x}{m}- \frac{\varphi m_x}{m^2}+\varphi(\beta+\varphi)\tanh\alpha \)e_4,
\\(\bar\nabla_{e_2}h)(e_1,e_1)=&\(\frac{\beta_y}{m}+ \frac{\beta m_y}{m^{2}} -\beta(\delta+\mu) \tanh \alpha \)e_3\\& \hskip.1in +\(\frac{\gamma_y}{m}+\frac{3\gamma m_y}{m^2}+\gamma(\delta+\mu)\tanh \alpha\)e_4
,  \\ (\bar\nabla_{e_1}h)(e_2,e_2)=& \(\frac{\lambda_x}{m}+\frac{3\lambda m_x}{m^{2}} - \lambda(\beta+\varphi)\tanh \alpha\)e_3 \\& \hskip.1in + \(\frac{\mu_x}{m}+\frac{\mu m_x}{m^{2}} +\mu(\beta+\varphi) \tanh \alpha \)e_4,
\\ (\bar\nabla_{e_2}h)(e_1,e_2)=&\(\frac{\delta_y}{m}-\frac{\delta m_y}{m^2}-\delta(\delta+\mu)\tanh\alpha  \)e_3 \\& \hskip.1in +\(\frac{\varphi_y}{m}+ \frac{\varphi m_y}{m^2}+\varphi(\delta+\mu)\tanh\alpha \)e_4
,\\ (\bar\nabla_{e_2}h)(e_2,e_2)=&\(\frac{\lambda_y}{m}-\frac{3\lambda m_y}{m^2}-\lambda(\delta+\mu)\tanh\alpha  \)e_3
\\& \hskip.1in +\(\frac{\mu_y}{m}- \frac{\mu m_y}{m^2}+\mu(\delta+\mu)\tanh\alpha \)e_4.
\end{aligned}\end{equation} 

On the other hand, from  \e{3.10} we also find

\begin{equation}\begin{aligned}\label{3.23}&( \tilde R(e_1,e_2)e_2)^\perp=-\sech \alpha \tilde R(e_1,e_2;e_2,Je_2)e_3\\&\hskip.5in -\{\tanh \alpha \tilde K+
\sech \alpha  \tilde R(e_1,e_2;e_2,Je_1)\}e_4 ,
\\& ( \tilde R(e_2,e_1)e_1)^\perp=\{\tanh \alpha \tilde K-\sech \alpha  \tilde R(e_2,e_1;e_1,Je_2)\}e_3  \\&\hskip.5in -\sech \alpha  \tilde R(e_2,e_1;e_1,Je_1) e_4.\end{aligned}\end{equation} 
By applying \e{3.4}, \e{3.12}, \e{3.22}, \e{3.23}, and the equation of Codazzi we get

\begin{equation}\begin{aligned}\label{3.24}& \lambda_x-\delta_y =( \lambda\beta+\lambda\varphi-\delta^2-\delta\mu)m \tanh\alpha-\frac{\delta m_y+3\lambda m_x}{m}\\& \hskip.6in -m\sech\alpha  \tilde R(e_1,e_2;e_2,Je_2),
\\& \mu_x-\varphi_y=(\delta\varphi -\beta\mu) m\tanh \alpha +\frac{\varphi m_y-\mu m_x}{m}
\\& \hskip.6in -m\sech\alpha  \tilde R(e_1,e_2;e_2,Je_1) -m (\tanh \alpha) \tilde K,
\\& \beta_y -\delta_x=(\beta\mu-\delta\varphi) m\tanh \alpha+\frac{\delta m_x-\beta m_y}{m}
\\& \hskip.6in -m\sech\alpha \tilde R(e_2,e_1;e_1,Je_2)+m( \tanh \alpha) \tilde K ,
\\& \gamma_y-\varphi_x =(\beta\varphi+\varphi^2-\delta\gamma-\gamma\mu) m\tanh \alpha -\frac{\varphi m_x+3\gamma m_y}{m}\\& \hskip.6in -m\sech\alpha  \tilde R(e_2,e_1;e_1,Je_1).
\end{aligned}\end{equation} 
Also, from \e{3.4}, \e{3.5}, \e{3.15}, \e{3.17} and  \e{3.19}  we have
\begin{align} & \label{3.25} A_{e_3}=\begin{pmatrix} \varphi &\mu \\ \gamma &\varphi\end{pmatrix},\; A_{e_4}=\begin{pmatrix} \delta &\lambda \\ \beta &\delta \end{pmatrix},
\\ \label{3.26} &\alpha_x=m(\varphi-\beta) ,\;\; \alpha_y=m(\mu-\delta).
\end{align} 

By applying \e{3.10}, \e{3.11} and \e{3.25} we derive that
\begin{align}\label{3.27}& \tilde R(e_1,e_2;e_3,e_4)=(\sech^2\alpha-\tanh^2\alpha) \tilde K
\\&\notag \hskip.5in  -2\sech \alpha\tanh\alpha \tilde R(e_1,e_2; e_2,Je_1),
\\\label{3.28} &\<[A_{e_3},A_{e_4}]e_1,e_2\>=\gamma \lambda-\beta\mu.\end{align}

From \e{3.6}, \e{3.21}, and \e{3.28},  we find
\begin{equation}\begin{aligned}\label{3.29} &\hskip-.2in \tilde g(R^D(e_1,e_2)e_3,e_4)=\frac{2m m_{xy}-2m_xm_y}{m^4}\\&+\left\{ (\delta+\mu)\alpha_x-(\beta+\varphi)\alpha_y\right\}\frac{\sech^2\alpha}{m}\\&\hskip-.1in +\{(\delta+\mu)m_x-(\beta+\varphi)m_y+m(\delta_x+\mu_x-\beta_y-\varphi_y)\}\frac{\tanh\alpha}{m^2}.\end{aligned}\end{equation}
Therefore,  the equation of Ricci is given by
\begin{equation}\begin{aligned}\label{3.30} &\hskip.1in \frac{2m m_{xy}-2m_xm_y}{m^4}+\left\{ (\delta+\mu)\alpha_x-(\beta+\varphi)\alpha_y\right\}\frac{\sech^2\alpha}{m}\\& +\{(\delta+\mu)m_x-(\beta+\varphi)m_y+m(\delta_x+\mu_x-\beta_y-\varphi_y)\}\frac{\tanh\alpha}{m^2}\\&\hskip-.1in  = \gamma \lambda-\beta\mu +(\sech^2\alpha-\tanh^2\alpha) \tilde K -2\sech \alpha\tanh\alpha \tilde R(e_1,e_2;e_2,Je_1)
.\end{aligned} \end{equation}

On the other hand, using \e{3.4} and \e{3.17} we find
\begin{align} & \label{3.31} (\delta+\mu)\alpha_x-(\beta+\varphi)\alpha_y=2m(\delta\varphi-\beta\mu).
\end{align}
Also, by applying \e{3.24}, we get
\begin{equation}\begin{aligned}\label{3.32} &(\delta+\mu)m_x-(\beta+\varphi)m_y+m(\delta_x+\mu_x-\beta_y-\varphi_y)
\\&\hskip.1in =2(\delta\varphi -\beta\mu) m^2\tanh \alpha -2m^2 \tanh \alpha \tilde K
 \\& \hskip.2in +m^2\sech\alpha\big\{R(e_2,e_1;e_1,Je_2) -\tilde R(e_1,e_2;e_2,Je_1)\big\}.\end{aligned} \end{equation}
Substituting \e{3.31} and \e{3.32} into equation \e{3.30} gives
\begin{equation}\begin{aligned}\label{3.33} &
  \gamma \lambda+\beta\mu-2\delta\varphi=\frac{2m m_{xy}-2m_xm_y}{m^4}-\tilde K  \\& \hskip.1in 
  -{\tanh\alpha}\sech\alpha\big\{\tilde R(e_2,e_1;e_1,Je_2) +\tilde R(e_1,e_2;e_2,Je_1)\big\} 
.\end{aligned} \end{equation}

On the other hand, by applying the curvature identities \e{1.3} and \e{1.5}, we find
$$ \tilde R(e_2,e_1;e_1,Je_2)=- \tilde R(e_1,e_2;e_2,Je_1).$$ 
Combining this with  \e{3.33}  shows that equation \e{3.33} becomes equation \e{3.20} of Gauss. Consequently, the equation of Ricci is a consequence  of Gauss and Codazzi for arbitrary Lorentz  surfaces in any Lorentzian Kaehler surface. 
\end{proof}

From the proof of Theorem 1 we also have the following. 

\begin{theorem} The equation of Gauss is a consequence of  the equations of Codazzi and Ricci for  Lorentz surfaces in Lorentzian Kaehler surfaces.
\end{theorem}

\begin{remark} Some special cases of Theorem 1 are obtained in \cite{c2,c3}.
\end{remark} 

\begin{remark}  Theorem 1 is false in general if the Lorentz surface in a Lorentzian Kaehler surface were replaced by a spatial surface in a Lorentzian Kaehler surface.
\end{remark}

\begin{remark} Since the three fundamental equations of Gauss, Codazzi and Ricci provide the conditions for local isometric embeddability, these equations also play some important role in physics; in particular in the Kaluza-Klein theory (cf. \cite{hv,M,S}).
\end{remark}

\end{document}